\documentclass{amsart} 
\usepackage{amssymb,times, amscd,amsmath,amsthm, xypic}

\author{Shoji Yokura$^{(*)}$}
\address
{Department of Mathematics and Computer Science, 
Faculty of Science, 
Kagoshima University, 21-35 Korimoto 1-chome, Kagoshima 890-0065, Japan}
\email {yokura@sci.kagoshima-u.ac.jp}
\title
{Motivic Milnor classes}
\thanks {(*) Partially supported by Grant-in-Aid for Scientific Research
(No. 21540088), the Ministry of Education, Culture, Sports, Science and Technology (MEXT), and JSPS Core-to-Core Program 18005, Japan}

\begin{document} 
\numberwithin{equation}{section}
\newtheorem{thm}[equation]{Theorem}
\newtheorem{pro}[equation]{Proposition}
\newtheorem{prob}[equation]{Problem}
\newtheorem{cor}[equation]{Corollary}
\newtheorem{con}[equation]{Conjecture}
\newtheorem{lem}[equation]{Lemma}
\theoremstyle{definition}
\newtheorem{ex}[equation]{Example}
\newtheorem{defn}[equation]{Definition}
\newtheorem{rem}[equation]{Remark}
\renewcommand{\rmdefault}{ptm}
\def\alp{\alpha}
\def\be{\beta}
\def\jeden{1\hskip-3.5pt1}
\def\om{\omega}
\def\bigstar{\mathbf{\star}}
\def\ep{\epsilon}
\def\vep{\varepsilon}
\def\Om{\Omega}
\def\la{\lambda}
\def\La{\Lambda}
\def\si{\sigma}
\def\Si{\Sigma}
\def\Cal{\mathcal}
\def\m {\mathcal}
\def\ga{\gamma}
\def\Ga{\Gamma}
\def\de{\delta}
\def\De{\Delta}
\def\bF{\mathbb{F}}
\def\bH{\mathbb H}
\def\bPH{\mathbb {PH}}
\def \bB{\mathbb B}
\def \bA{\mathbb A}
\def \bOB{\mathbb {OB}}
\def \bM{\mathbb M}
\def \bOM{\mathbb {OM}}
\def \calB{\mathcal B}
\def \bK{\mathbb K}
\def \bG{\mathbf G}
\def \bL{\mathbf L}
\def\bN{\mathbb N}
\def\bR{\mathbb R}
\def\bP{\mathbb P}
\def\bZ{\mathbb Z}
\def\bC{\mathbb  C}
\def \bQ{\mathbb Q}
\def\op{\operatorname}

\begin{abstract} The Milnor class is a generalization of the Milnor number, defined as the difference (up to sign) of Chern--Schwartz--MacPherson's class and Fulton--Johnson's canonical Chern class of a local complete intersection variety in a smooth variety. In this paper we introduce a ``motivic" Grothendieck group $K^{\mathcal Prop} _{\ell.c.i}(\mathcal V/X \xrightarrow{h}  S)$ and natural transformations from this Grothendieck 
group to the homology theory. We capture the Milnor class, more generally Hirzebruch--Milnor class, as a special value of a distinguished element under these natural transformations. We also show a Verdier-type Riemann--Roch formula for our motivic Hirzebruch-Milnor class. We use Fulton--MacPherson's bivariant theory and the motivic Hirzebruch class.
\end{abstract}

\maketitle

\section{Introduction}\label{intro} 

The Milnor class is defined for a local complete intersection variety $X$ in a non-singular variety $M$ as follows.
The local complete intersection variety $X$ defines a normal bundle $N_X$ in $M$, from which we can define the virtual tangent bundle $T_X$ of $X$ by
$$T_X := TM|_X - N_XM$$
which  is a well-defined element of the Grothendieck group $K^0(X)$.  
Then Fulton-Johnson's or Fulton's canonical (Chern) class of $X$ (see \cite{Fulton-Johnson} and \cite{Fulton-book}) is defined by 
$$c_*^{FJ}(X) := c(T_X) \cap [X].$$
Here $c(T_X)$ is the total Chern class of the virtual bundle $T_X$.

In general, Fulton-Johnson's and Fulton's canonical (Chern) classes  are defined for any scheme $X$ embedded as a closed subscheme of a non-singular variety $M$(see \cite[Example 4.2.6]{Fulton-book}): Fulton--Johnson's canonical class $c_*^{FJ}(X)$ (\cite[Example 4.2.6 (c)]{Fulton-book}) is defined by
$$c(TM|_X) \cap s(\m N_XM),$$
where $TM$ is the tangent bundle of $M$ and $s(\m N_XM)$ is the Segre class of the conormal sheaf $\m N_XM$ of $X$ in $M$ \cite[\S 4.2]{Fulton-book}. Fulton's canonical class $c_*^F(X)$ (\cite[Example 4.2.6 (a)]{Fulton-book}) is defined by
$$c(TM|_X) \cap s(X,M),$$
where $s(X,M)$ is the relative Segre class \cite[\S 4.2]{Fulton-book}.
As shown in \cite[Example 4.2.6]{Fulton-book}, for a local complete intersection variety $X$ in a non-singular variety $M$ these two classes are both equal to $c(T_X) \cap [X]$. \\

On the other hand there is another well-known notion of Chern class for possibly singular varieties. That is Chern--Schwartz--MacPherson's class $c_*(X)$ \cite{MacPherson1, Schwartz1, Schwartz2, Schwartz3, Brasselet-Schwartz}. Then the Milnor class of the local complete intersection variety $X$, denoted by $\m M(X)$, is defined by, up to sign, the difference of Fulton--Johnson's class and Chern--Schwarz--MacPherson's class $c_*(X)$; more precisely
$$\m M(X) := (-1)^{\op {dim}X}\left (c_*^{FJ}(X) - c_*(X) \right ).$$
Since Chern--Schwarzt--MacPherson's class $c_*(X)$ and Fulton--Johnson's class $c_*^{FJ}(X)$ are identical for a nonsingular variety, the Milnor class is certainly supported on the singular locus of the given variety, thus is an invariant of singularities. Prototypes of the Milnor class were studied by P. Aluffi \cite {Aluffi-Duke, Aluffi}, A. Parusi\'nski \cite{Parusinski, Parusinski-London}, A. Parusi\'nski  and P. Pragacz \cite{Parusinski-Pragacz-JAG0} and T. Suwa \cite{Suwa3}. 
Many people have been investigating on the Milnor class from their own viewpoints or interests, and many papers are now available \cite{Aluffi, Aluffi-lecture, Brasselet, BLSS1, BLSS2, Maxim, Parusinski-lecture-note, Parusinski-Pragacz-JAMS, Parusinski-Pragacz-JAG, Seade, SeSu, Suwa-lecture-note, Yokura-VRR-Chern, Yokura-contemporary}. A category-functorial aspect of the Milnor class is its connection to the so-called Verdier--Riemann--Roch theorem for MacPherson's Chern class \cite{Yokura-topology, Schuermann-VRR}. 

Some functoriality of the Milnor class was investigated in \cite{Yokura-topology}, but so far it has never been captured as a natural transformation from a certain covariant functor to the homology theory. In this paper we try to capture the Milnor class from a motivic viewpoint and we show that in fact we can capture it as a natural transformation from a pre-motivic covariant functor to
the homology theory. For this we need to use the motivic Hirzebruch class \cite{BSY1, BSY2}. The key idea comes from the construction of a universal bivariant theory given in \cite{Yokura-obt}.\\

In \cite{CMSS} (also see  \cite{CLMS1, CLMS2, CMS1, CMS2, CS2, CS3}) Sylvain Cappell et al. independenetly consider the motivic Hirzebruch--Milnor class and they describe it in terms of other invariants of singularities, thus dealing more with singularities.  Our present work is more category-functorial, compared with \cite{CMSS}.
 
 \section{Motivic Hirzebruch classes}
 In the following sections we use the motivic Hirzebruch class \cite{BSY1, BSY2}, thus we very quickly recall some ingredients which are needed later.
 
 Let $\m V$ denote the category of complex algebraic varieties. The relative Grothendieck group $K_0(\Cal V / X)$ of a variety $X$ is the quotient of the free abelian group $\op{Iso}^{\Cal Prop}(\m V/X) $ of isomorphism classes $[V  \xrightarrow {h} X]$ of proper morphisms to $X$ , modulo the following {\it additivity} relation:
$$[V   \xrightarrow {h}  X] = [Z \hookrightarrow V  \xrightarrow {h}  X] +  [V \setminus Z \hookrightarrow Y   \xrightarrow {h} X]$$
for $Z \subset Y$ a closed subvariety of $Y$. We set the quotient homomorphism by
$$\Theta: \op{Iso}^{\Cal Prop}(X) \to K_0(\Cal V / X).$$
From now on the equivalence class $\Theta([V   \xrightarrow {h}  X])$ of the isomorphism class $[V   \xrightarrow {h}  X]$ is denoted by the same symbol $[V   \xrightarrow {h}  X]$ unless some possible confusion occurs.

\begin{rem}Furthermore it follows from Hironaka's resolution of singularities that the restriction $\Theta ^{sm}:= \Theta |_{\op{Iso}^{\Cal Prop}(Sm/X)}$ of $\Theta$ to the subgroup $\op{Iso}^{\Cal Prop}(Sm/X)$ of  isomorphism classes $[V  \xrightarrow {h} X]$ of proper morphisms from \underline {smooth varieties} $V$ to $X$ is surjective: 
$$\Theta^{sm} : \op{Iso}^{\Cal Prop}(Sm/X) \to  K_0( \Cal V /X).$$
Here we just remark that F. Bittner \cite {Bittner} identified the kernel of the above  map $\Theta^{sm}: \op{Iso}^{\Cal Prop}(Sm/X) \to  K_0( \Cal V /X)$ by some ``blow-up relation", for the details of which see \cite {Bittner}.
This ``blow-up relation" plays an important role for constructing a bivariant analogue of the motivic Hirzebruch classes.
Since we do not deal with this bivariant analogue, we do not go further into details of this ``blow-up relation".\\
\end{rem}

If we use the above ``pre-motivic" group $\op{Iso}^{\Cal Prop}(Sm/X) $ we can get the following 

\emph{``pre-motivic" characteristic classes of singular varieties for an arbitrary characteristic class $c\ell$ of complex vector bundles}.

For a proper morphism $f: X \to Y$ we have the obvious pushforward
$$f_* :\op{Iso}^{\Cal Prop}(Sm/X)  \to \op{Iso}^{\Cal Prop}(Sm/Y)$$
defined by $f_*([ V   \xrightarrow {h}  X]) := [V   \xrightarrow {f \circ h}  Y].$ Let $c\ell$ be any characteristic class of complex vector bundles with values  in the cohomology theory $H^*(\quad)\otimes R$. Then we define 
$$\ga_{c\ell}: \op{Iso}^{\Cal Prop}(Sm/X) \to H_*^{BM}(X) \otimes R$$ by
$$\ga_{c\ell}([ V   \xrightarrow {h}  X]):= h_*(c\ell(TV) \cap [V]).$$
Then it is clear that 
$$\ga_{c\ell}:\op{Iso}^{\Cal Prop}(Sm/\quad) \to H_*^{BM}(\quad) \otimes R$$
is a unique natural transformation satisfying the normalization condition that for a smooth variety $X$ the homomorphism $\ga_{c\ell}: \op{Iso}^{\Cal Prop}(Sm/X) \to H_*^{BM}(X) \otimes R$ satisfies that
$$\ga_{c\ell}([ X   \xrightarrow {\op {id}_X}  X]):= c\ell(TX) \cap [X].$$
A na\" \i ve question is whether $\ga_{c\ell}$ can be pushed down to the relative Grothendieck group $K_0(\Cal V / X)$ , i.e., for some natural transformation $?: K_0(\Cal V / X) \to H_*^{BM}(\quad) \otimes R$ so that the following diagram commutes:

$$
\xymatrix{
& \op{Iso}^{\Cal Prop}(Sm/X)  \ar [dl]_{\Theta^{sm}} \ar [dr]^{\ga_{c\ell}} \\
{K_0(\m V/X)} \ar [rr] _{?}& &  H_*^{BM}(X)\otimes R \,}
$$

If we require that $c\ell$ is a multiplicative characteristic class, the above normalization condition and another extra condition that the degree of the $0$-dimensional component of the class $\ga_{c\ell}(\mathbb {CP}^n)$ equals $1 - y + y^2 + \cdots (-y)^n$, then the characteristic class $c\ell$ can be identified as the Hirzebruch class. Namely, let $\alp _i$'s be the Chern roots of  a complex vector bundle $E$ over $X$. Then 
$$td(E) = \prod _{i=1}^{\op {rank} E} \frac{\alp _i}{1-e^{-\alp _i}} \in H^{2*}(X;\bQ)$$
is the \emph {Todd class} of $E$, and its modified version of it
$$td_{(y)}(V) := \prod _{i=1}^{\op {rank} E} \left (\frac {\alp _i(1+y)}{1-e^{-\alp _i(1+y)}} - \alp _i y \right ) \in H^*(X) \otimes \bQ[y]$$\\
is called the \emph {Hirzebruch class} (see \cite{Hirzebruch} and \cite{HBJ}. In fact, the Hirzebruch class unifies Chern, Todd  and Thom--Hirzebruch classes:
\begin{enumerate}
\item 
\underline {$y = -1$}: $ td_{(-1)}(E) = c(E)$ \, Chern class, 
\item 
\underline {$y = 0$}: $td_{(0)}(E) = td(E)$ \, Todd class, 
\item 
\underline {$y = 1$}: $td_{(1)}(E)  = L(E)$ \, Thom--Hirzebruch  $L$-class.\\
\end{enumerate}

Our previous paper \cite{BSY1} (also see \cite{BSY2} and \cite{Schuermann-Yokura}) showed the following theorem 
(originally using Saito's theory of mixed Hodge modules \cite{Saito}):

\begin{thm} (Motivic Hirzebruch class of singular varieties)\label{motivic-Hirzebruch}
There exists a unique natural transformation
$${T_y}_* : K_0(\m V/\quad)  \to H_*^{BM}(\quad) \otimes \bQ[y]$$
satisfying the normalization condition that for a smooth variety $X$
$${T_y}_* ([X  \xrightarrow{\op {id}_X}  X]) = td_{(y)}(TX) \cap [X].$$
\end{thm}

This motivic Hirzebruch class ${T_y}_* : K_0(\m V/\quad)  \to H_*^{BM}(\quad) \otimes \bQ[y]$ in a sense ``unifies the following three well-known characteristic classes of singular varieties:

\begin{thm}(A ``unification" of three characteristic classes)\label{unification}
\begin{enumerate}
\item \underline {$c$ = Chern class}:  There exists a unique natural transformation 
$$\ga_F: K_0(\Cal V/\quad) \to F(\quad)$$
 such that for $X$ nonsingular $\ga_F([X \xrightarrow {\op {id}} X]) = \jeden_X.$ And the following diagram commutes
$$\xymatrix{
& K_0(\Cal V/X)  \ar [dl]_{\ga_F} \ar [dr]^{{T_{-1}}_*} \\
{F(X) } \ar [rr] _{c_*}& &  H_*^{BM}(X)\otimes \bQ.}
$$
Here $c_*:F(X) \to H_*^{BM}(X)$ is  the rationalized MacPherson's Chern class transformation \cite{MacPherson1}.
\item \underline {$td$ = Todd class}: There exists a unique natural transformation 
$$\ga_{G_0}: K_0(\Cal V/\quad) \to G_0(\quad)$$
 such that for $X$ nonsingular $\ga([X \xrightarrow {\op {id}} X]) = [\Cal O_X].$ And the following diagram commutes
$$\xymatrix{
&  K_0(\Cal V/X)  \ar [dl]_{\ga_{G_0}} \ar [dr]^{{T_{0}}_*} \\
{G_0(X) } \ar [rr] _{td_*}& &  H_*^{BM}(X)\otimes \bQ.}
$$
Here $td_*:G_0(X) \to H_*^{BM}(X)\otimes \bQ$ is Baum--Fulton--MacPherson's Todd class (or Riemann--Roch) transfomration \cite{BFM}.
\item  \underline {$L$ = Thom-Hirzebruch $L$-class}:There exists a unique natural transformation 
$$\ga_{\Omega} :K_0(\Cal V/\quad) \to \Omega (\quad)$$
such that for $X$ nonsingular $\omega([X \xrightarrow {\op {id}} X]) = \left [\bQ_X[\op {dim}X] \right ].$ And the following diagram commutes
$$\xymatrix{
& K_0(\Cal V/X)  \ar [dl]_{\ga_{\Omega}} \ar [dr]^{{T_{1}}_*} \\
{\Omega(X) } \ar [rr] _{L_*}& &  H_*^{BM}(X)\otimes \bQ.}
$$
Here $\Omega(X)$ is the Cappell--Shaneson--Youssin's cobordism group of self-dual constructible sheaves (see \cite{CS} and \cite {You}) and  $L_*:\Omega(X) \to H_*^{BM}(X)\otimes \bQ$ is Cappell--Shaneson's homology L-class transformation \cite{CS} (also see \cite{GM}).
\end{enumerate}
\end{thm}

We also have the following
\begin{cor} The following diagram commutes:
$$
\xymatrix{
& \op{Iso}^{\Cal Prop}(Sm/X)  \ar [dl]_{\Theta^{sm}} \ar [dr]^{\ga_{td_{(y)}}} \\
{K_0(\m V/X)} \ar [rr] _{{T_y}_*}& &  H_*^{BM}(X)\otimes R \,}
$$
\end{cor}

\begin{defn} For a complex algebraic variety $X$
$${T_y}_*(X) := {T_y}_*([X \xrightarrow {\op {id}} X]) \in H_*^{BM}(X) \otimes \bQ[y]$$
is called \emph{the motivic Hirzebruch class of $X$}.
\end{defn}

\begin{rem} \label{DuBois} As to the homomorphism $\ga_F: K_0(\Cal V/X) \to F(X)$ we have that for \emph{any} variety $X$
$$\ga_F( [X \xrightarrow {\op {id}} X]) = \jeden_X, \text {therefore} \quad {T_{-1}}_*(X) = c_*(X),$$
whether $X$ is singular or non-singular.  However, as to the other two homomorphisms $\ga_{G_0}: K_0(\Cal V/X) \to G_0(X)$ and  $\ga_{\Omega} :K_0(\Cal V/X) \to \Omega (X)$,  if $X$ is singular, in general we have that  
$$\ga_{G_0}( [X \xrightarrow {\op {id}} X]) \not = [\m O_X], \quad \ga_{\Omega}( [X \xrightarrow {\op {id}} X]) \not = [\m {IC}_X],$$
where $\m {IC}_X$ is the middle intersection homology complex of Goresky--MacPherson \cite{GM}.
Hence,  if $X$ is singular, in general we have that 
$${T_{0}}_*(X) \not = td_*(X), \quad {T_{1}}_*(X) \not = L_*(X).$$
If $X$ is a Du Bois variety, i.e., a variety with Du Bois singularities, then we have that 
$$\ga_{G_0}( [X \xrightarrow {\op {id}} X]) = [\m O_X], \text {therefore} \quad {T_{0}}_*(X) = td_*(X).$$
  If $X$ is a rational homology manifold, then conjecturally 
$$\ga_{\Omega}( [X \xrightarrow {\op {id}} X]) = [\m {IC}_X], \text {therefore} \quad {T_{1}}_*(X) = L_*(X).$$ For more details, see \cite{BSY1}.
\end{rem}

\section {The Grothendieck group $K^{\Cal Prop} _{\ell.c.i}(\m V/X \xrightarrow{h}  S)$}

Let $S$ be a complex algebraic variety and fixed. Let $\m V_S$ be the category of $S$-varieties, i.e., an object  is a morphism $h: X \to S$ and a morphism from $h: X \to S$ to $k: Y \to S$  is a morphism $f:X \to Y$ such that the following diagram commutes:
$$\xymatrix{
X \ar[dr]_ {h}\ar[rr]^ {f} && Y \ar[dl]^{k }\\
& S}.$$

\begin{defn}
Let $M^{\Cal Prop} _{\ell.c.i}(\m V/X \xrightarrow{h}  S)$ be the monoid consisting of isomorphism classes $[V \xrightarrow{p}  X]$ of \underline {proper} morphisms $p: V \to X$ such that the composite $h \circ p: V \to S$ becomes a \underline {local complete intersection} morphism, with the addition $(+)$ and zero $(0)$ defined by
\begin{itemize}
\item $[V  \xrightarrow{h}  X] + [V'  \xrightarrow{h'}  X] := [V \sqcup V' \xrightarrow{h + h'} X]$,
\item $0 := [\phi \to X]. $
\end{itemize}
Then we define
$$K^{\Cal Prop} _{\ell.c.i}(\m V/X \xrightarrow{h}  S)$$
to  be the Grothendieck group of the monoid $M^{\Cal Prop} _{\ell.c.i}(\m V/X \xrightarrow{h}  S)$.\\
 \end{defn}
 
\begin{lem} 
\begin{enumerate}
\item The Grothendieck group $K^{\Cal Prop} _{\ell.c.i}(\m V/X \xrightarrow{h}  S)$ is a covariant functor with pushforwards for proper morphisms, i.e., for a proper morhism $f: X \to Y \in \m V_S$ 
$$\xymatrix{
X \ar[dr]_ {h}\ar[rr]^ {f} && Y \ar[dl]^{k }\\
& S}$$
the pushforward 
$$f_*: K^{\Cal Prop} _{\ell.c.i}(\m V/X \xrightarrow{h}  S) \to  K^{\Cal Prop} _{\ell.c.i}(\m V/Y \xrightarrow{k}  S) $$defined by
$$f_*([V  \xrightarrow{p}  X]) := [V  \xrightarrow{f \circ p}  Y]$$
is covariantly functorial.
\item The Grothendieck group $K^{\Cal Prop} _{\ell.c.i}(\m V/X \xrightarrow{h}  S)$ is a contravariant functor with pullbacks for smooth morphisms, i.e., for a smooth morhism $f: X \to Y \in \m V_S$ 
the pullback
$$f^*: K^{\Cal Prop} _{\ell.c.i}(\m V/Y \xrightarrow{k}  S) \to  K^{\Cal Prop} _{\ell.c.i}(\m V/X \xrightarrow{h}  S) $$defined by
$$f^*([W  \xrightarrow{p}  Y]) := [W'  \xrightarrow{p'}  X]$$
is contravariantly functorial. Here we consider the following commutative diagrams whose top square is a fiber square:
\begin{equation}\label{diagram1}
\xymatrix{
W' \ar[rr]^{f'}\ar[d]_{p'} && W \ar[d]^{p}\\
X \ar[dr]_ {h}\ar[rr]^ {f} && Y \ar[dl]^{k }\\
& S.}
\end{equation}

\end{enumerate}
\end{lem}
\begin{rem} 

\begin{enumerate}

\item As to the contravariance of the Grothendieck group $K^{\Cal Prop} _{\ell.c.i}(\m V/X \xrightarrow{h}  S)$, one might be tempted to consider the pullback for a local complete intersection morphism $f:X \to Y$ instead of a smooth morphism. But a crucial problem for this is that the pullback of a local complete intersection morphism is not necessarily a local complete intersection morphism, thus in the above diagram (\ref{diagram1}) $f': W' \to W$ is not necessarily a local complete intersection morphism and hence we do not know whether or not the composite $k \circ p \circ f' = h \circ p'$ is a local complete intersection morphism.
\item If we consider the finer class $\m Sm$ of smooth morphisms instead of the class $\m L.c.i$ of local complete intersection morphisms, we do have a bivariant theory, from which we can construct a motivic bivariant characteristic class \cite{Yokura-mbc}.
\end{enumerate}
\end{rem}

\section{Motivic Hirzebruch-Milnor classes}

For a morphism $f:X \to Y$, $\bH(X \to Y)$ is the Fulton--MacPherson bivariant homology theory \cite{Fulton-MacPherson}. Since the main theme of the present paper is not a bivariant theoretic, we do not recall a general bivariant theory, thus see \cite{Fulton-MacPherson} for details. In the paper $\bullet$ denotes the bivariant product, i.e., for morphisms $f: X \to Y$, $g:Y \to Z$ the bivariant product $\bullet$ is 
$$\bullet: \bH(X \xrightarrow{f} Y) \times \bH(Y \xrightarrow{g} Z) \to \bH(X \xrightarrow{g\circ f} Z).$$
Then  $\bH(X  \xrightarrow{\op {id}_X}  X)$ is the usual cohomology theory $H^*(X)$ and $\bH(X \to pt)$ (for a mapping to a point) is the Borel--Moore homology theory $H_*^{BM}(X)$.

\begin{pro}Let $c\ell:K^0 \to H^*(\quad)\otimes R$ be a characteristic class of complex vector bundles with a suitable coefficients $R$. Then on the category $\m V_S$ we have that 
\begin{enumerate}
\item  There exists a unique natural transformation (not a Grotendieck transformation) 
$$\widetilde {{\ga_{c\ell}}_*} : K^{\Cal Prop}  _{\ell.c.i}(\m V/X \xrightarrow{h}  S) \to \bH(X \xrightarrow{h}  S) \otimes R$$
such that for a local complete intersection morphism $h: X \to S$
$$\widetilde {{\ga_{c\ell}}_*} ([X  \xrightarrow{\op {id}_X}  X]) = c\ell(T_h) \bullet U_h.$$
Here $T_h$ is the relative tangent bundle of $h$ and $U_h \in \bH(X \xrightarrow{h}  S)$ is the canonical orientation.
\item There exists a unique natural transformation 
$${\ga_{c\ell}}_* : K^{\Cal Prop}  _{\ell.c.i}(\m V/X \xrightarrow{h}  S) \to H_*^{BM}(X) \otimes R$$
such that for a local complete intersection morphism $h: X \to S$
$${\ga_{c\ell}}_*([X  \xrightarrow{\op {id}_X}  X]) = c\ell(T_h) \cap [X].$$
\end{enumerate}
\end{pro}

\begin{proof} 
(1) We define ${\ga_{c\ell}}_* : K^{\Cal Prop}  _{\ell.c.i}(\m V/X \xrightarrow{h}  S) \to \bH(X \xrightarrow{h}  S) \otimes R$ by 
$$\widetilde {{\ga_{c\ell}}_*} ([V  \xrightarrow{p}  X]) := p_*(c\ell(T_{h \circ p} \bullet U_{h \circ p}).$$
Then for a morphism $f:X \to Y$, i.e., for the following commutative diagram
$$\xymatrix{
X \ar[dr]_ {h}\ar[rr]^ {f} && Y \ar[dl]^{k }\\
& S}$$
the following diagram commutes:
$$\CD
K^{\Cal Prop}  _{\ell.c.i}(\m V/X \xrightarrow{h}  S) @> {\widetilde {{\ga_{c\ell}}_*} } >> \bH(X \xrightarrow{h}  S) \otimes R \\
@V f_*VV @VV f_* V\\
K^{\Cal Prop}  _{\ell.c.i}(\m V/Y \xrightarrow{k}  S)  @>> {\widetilde {{\ga_{c\ell}}_*} } > \bH(Y \xrightarrow{k}  S) \otimes R, \endCD
$$
Indeed, for $[V  \xrightarrow{p}  X] \in K^{\Cal Prop}  _{\ell.c.i}(\m V/X \xrightarrow{h}  S)$ we have that 
\begin{align*}
f_*\left (\widetilde {{\ga_{c\ell}}_*} ([V  \xrightarrow{p}  X])\right) & = f_* \left (p_*(c\ell(T_{h \circ p}) \bullet U_{h \circ p}) \right)\\
& = (f \circ p)_*\left(c\ell(T_{h \circ p}) \bullet U_{h \circ p}\right) \\
& = (f \circ p)_*\left(c\ell(T_{k \circ f \circ p}) \bullet U_{k \circ f  \circ p}\right) \\
& = (f \circ p)_*\left(c\ell(T_{k \circ f \circ p}) \bullet U_{k \circ f  \circ p}\right) \\
& = {\ga_{c\ell}}_*([V \xrightarrow{f \circ p} Y]) \\
& = \widetilde {{\ga_{c\ell}}_*}\left (f_*([V  \xrightarrow{p}  X]) \right). 
\end{align*}
Since, for a local complete intersection morphism $h: X \to S$, by definition of  ${\ga_{c\ell}}_*$ we have
${\ga_{c\ell}}_*([X  \xrightarrow{\op{id}_X}  X]) = c\ell(T_h) \bullet U_h$,
the uniqueness of ${\ga_{c\ell}}_*$ follows.

(2) We define ${\ga_{c\ell}}_* : K^{\Cal Prop}  _{\ell.c.i}(\m V/X \xrightarrow{h}  S) \to H_*^{BM}(X) \otimes R$ by 
$${\ga_{c\ell}}_*([V  \xrightarrow{p}  X]) := p_*(c\ell(T_{h \circ p} \cap [V]).$$
Then the following diagram commutes:
$$\CD
K^{\Cal Prop}  _{\ell.c.i}(\m V/X \xrightarrow{h}  S) @> {{\ga_{c\ell}}_*} >> H_*^{BM}(X) \otimes R \\
@V f_*VV @VV f_* V\\
K^{\Cal Prop}  _{\ell.c.i}(\m V/Y \xrightarrow{k}  S)  @>> {{\ga_{c\ell}}_*} > H_*^{BM}(Y) \otimes R, \endCD
$$
It follows from replacing $\bullet U_{h \circ p}$ and $\bullet U_{k \circ f \circ p}$  by $\cap [V]$ in the proof of (1).
\end{proof}

\begin{rem}
For a local complete intersection morphism $f: X \to S$, we have 
$$\bullet U_h \bullet [S] = \cap [X].$$ Here $[W]$ is the fundamental class of $W$ and $[W] \in \bH(W \to pt) = H_*^{BM}(W)$. Thus the relation between the above two natural transformations $\widetilde {{\ga_{c\ell}}_*}$ and ${\ga_{c\ell}}_*$ is that
$${\ga_{c\ell}}_* = \widetilde {{\ga_{c\ell}}_*} \bullet [S].$$
\end{rem}
\begin{rem}
When the fixed variety $S$ is a point, the above two natural transformations $\widetilde {{\ga_{c\ell}}_*}$ and ${\ga_{c\ell}}_*$ are the same: ${\ga_{c\ell}}_*: K^{\Cal Prop}  _{\ell.c.i}(\m V/X) \to H_*^{BM}(X) \otimes R.$
\end{rem}

If $S$ is a point and $c\ell = c$ the Chern class, then for a local complete intersection variety $X$ in a smooth manifold, we have that 
$${\ga_{c}}_*([X  \xrightarrow{\op{id}_X}  X]) = c(T_X) \cap [X]$$
which is Fulton--Johnson's class $c_*^{FJ}(X)$. Thus the above natural transformations  
$$\widetilde {{\ga_{c\ell}}_*} : K^{\Cal Prop}  _{\ell.c.i}(\m V/X \xrightarrow{h}  S) \to \bH(X \xrightarrow{h}  S) \otimes R$$
$${\ga_{c\ell}}_* : K^{\Cal Prop}  _{\ell.c.i}(\m V/X \xrightarrow{h}  S) \to H_*^{BM}(X) \otimes R$$
 are both generalizations of Fulton--Johnson's class as \underline {natural transformations}. They are respecively called \emph{a motivic ``bivariant" FJ-$c\ell$ class}, denoted by $\widetilde {c\ell_*^{FJ}}$, and \emph{a motivic FJ-$c\ell$ class}, denoted by $c\ell_*^{FJ}$, since it is modelled after Fulton--Johnson's class $c_*^{FJ}$.\\

From here on we consider the Hirzebruch class $td_{(y)}$, instead of an arbitrary characteristic class $c\ell$, because we use the motivic Hirzebruch class ${T_y}_*: K_0(\m V/X) \to H_*^{BM}(X) \otimes \bQ[y]$ below.  We use the above natural transformations 
$$\widetilde {{\ga_{td_{(y)}}}_*} : K^{\Cal Prop}  _{\ell.c.i}(\m V/X \xrightarrow{h}  S) \to \bH(X \xrightarrow{h}  S) \otimes \bQ[y],$$
$${\ga_{td_{(y)}}}_* : K^{\Cal Prop}  _{\ell.c.i}(\m V/X \xrightarrow{h}  S) \to H_*^{BM}(X) \otimes \bQ[y],$$
which are respectively called \emph{the motivic ``bivariant" FJ-Hirzebruch class} and  \emph{the motivic FJ-Hirzebruch class} and denoted by $\widetilde{{T_y}_*^{FJ}}$ and ${T_y}_*^{FJ}$.\\

We define the twisted pushforward for homology as follows: for a morphism $f:X \to Y$, the relative dimension of $f$ and the co-relative dimension of $f$ are respectively defined by 
$$\op {dim}(f) : = \op {dim}X - \op{dim}Y \quad \op{codim}(f ):= \op {dim}Y - \op{dim}X.$$
For the Borel--Moore homology theory $H_*$, the twisted pushforward for a proper morphism $f: X \to Y$ is define by 
$$f_{**} := (-1)^{\op{codim}(f)} f_*: H^{BM}_*(X) \to H^{BM}_*(Y).$$
With this twisted pushforward the Borel--Moore homology theory is still a covariant functor. To avoid a possible confusion we denote $H^{BM}_{**}(X)$ for the Borel--Moore homology theory with the twisted pushforward.

\begin{cor} On the category $\m V_S$ there exists a unique natural transformation
$$\m M {T_y}_*: K^{\Cal Prop} _{\ell.c.i}(\m V/X \xrightarrow{h}  S)  \to H_{**}^{BM}(X)\otimes \bQ[y]$$
such that for a local complete intersection morphism $h: X \to S$ the homomorphism $\m M {T_y}_*: K^{\Cal Prop} _{\ell.c.i}(\m V/X \xrightarrow{h}  S)  \to H_{**}^{BM}(X)\otimes \bQ[y]$ satisfies that 
$$\m M {T_y}_*([X  \xrightarrow{\op{id}_X}  X]) = (-1)^{\op{dim}X}\left ({T_y}_*^{FJ} - {T_y}_* \circ \Theta \right )([X  \xrightarrow{\op{id}_X} X]).$$
\end{cor}

\begin{proof} We define $\m M {T_y}_*: K^{\Cal Prop} _{\ell.c.i}(\m V/X \xrightarrow{h}  S)   \to H_{**}^{BM}(X)\otimes \bQ[y]$ by
$$\m M {T_y}_* ([V  \xrightarrow{p}  X]) :=  (-1)^{\op{dim}V} \left ({T_y}_*^{FJ}  - {T_y}_* \circ \Theta  \right )([V  \xrightarrow{p}  X]).$$
Which is equal to
$$(-)^{\op{dim}X} p_* \left (td_{(y)}(T_{p\circ h}) \cap [V] - {T_y}_*(V) \right ).$$
\end{proof}

From here on we denote ${T_y}_* \circ \Theta$ simply by ${T_y}_*$. The above motivic natural transformation 
$$\m M {T_y}_*: K^{\Cal Prop} _{\ell.c.i}(\m V/X)  \to H_{**}^{BM}(X)\otimes \bQ[y]$$
 shall be called \emph{a motivic Hirzebruch-Milnor class}, even though $K^{\Cal Prop} _{\ell.c.i}(\m V/X)$ is not (a subgroup of ) the motivic group $K_0(\m V/X)$, but because it is defined by using the motivic Hirzebruch class ${T_y}_*: K_0(\m V/X) \to H_*^{BM}(X) \otimes \bQ[y]$ and  because, if we specailze $\m M {T_y}_*$ to  the case when $y = -1$ and $X$ is a local complete intersection variety in a smooth manifold, we have 
 \begin{align*}
& \m M {T_{-1}}_*([X  \xrightarrow{\op{id}}  X]) \\
 & = (-1)^{\op{dim}X}\left \{td_{(-1)}(T_X)\cap [X]  - {T_{-1}}_* \left (\Theta ([X  \xrightarrow{\op{id}}  X]) \right ) \right \}\\
 & = (-1)^{\op{dim}X}\left (c^{FJ}_*(X) - c_*(X) \right ),
\end{align*}
 which is the Milnor class $\m M(X)$ of $X$. Thus $\m M {T_{-1}}_*: K^{\Cal Prop} _{\ell.c.i}(\m V/X )  \to H_{**}^{BM}(X)\otimes \bQ[y]$ is called \emph{the motivic Milnor class (or Chern--Milnor class)}. The more general one
 $$\m M {T_y}_*: K^{\Cal Prop} _{\ell.c.i}(\m V/X \xrightarrow{h}  S)  \to H_{**}^{BM}(X)\otimes \bQ[y]$$
 is called \emph{a motivic generalized Hirzebruch-Milnor class}.
 
 In fact, if the base variety $S$ is  a $\bQ$-homology manifold or a rational homology manifold, the fundamental class $[S] \in \bH(S \to pt) = H_*^{BM}(S)$ is a strong orientation (see \cite[Part I, \S 2.6]{Fulton-MacPherson}), namely we have the following isomorphism (see \cite{BSY3})
 $$\bullet [S] : \bH (X \xrightarrow{h}  S)\otimes \bQ \overset{\cong} \to \bH(X \to pt)\otimes \bQ =H_*^{BM}(X)\otimes \bQ.$$
Which is a generalized Poincar\'e duality isomorphism, hence denoted by $\m {PD}_h$. Indeed, when $X$ is a rational homology compact manifold, for the identity $id_X: X \to X$, the above isomorphism is nothing but the classical Poincar\'e duality isomorphism 
$$\cap [X]: H^*(X)\otimes \bQ \to H_*(X)\otimes \bQ.$$ 

Examples of a $\bQ$-homology manifold (e.g., see \cite[\S1.4 Rational homology manifolds]{Borho-MacPherson}) are surfaces with Kleinian sigularites, the moduli space of curves of a given genus, Satake's $V$-manifolds or orbifolds, in particular, the quotient of a nonsingular variety by a finite group action on.

Thus we can get the following corollary:
 
 \begin{cor} Let the base variety $S$ be a $\bQ$-homology manifold. On the category $\m V_S$ there exists a unique natural transformation
$$\widetilde {\m M {T_y}_*}: K^{\Cal Prop} _{\ell.c.i}(\m V/X \xrightarrow{h}  S)  \to \bH_{**}(X \xrightarrow{h}  S)\otimes \bQ[y]$$
such that for a local complete intersection morphism $h: X \to S$ the homomorphism $\widetilde {\m M {T_y}_*}: K^{\Cal Prop} _{\ell.c.i}(\m V/X \xrightarrow{h}  S)  \to \bH(X \xrightarrow{h}  S)\otimes \bQ[y]$ satisfies that 
$$\widetilde {\m M {T_y}_*}([X  \xrightarrow{\op{id}_X}  X]) = (-1)^{\op{dim}X}\left (\widetilde {{T_y}_*^{FJ}} - \m {PD}_h^{-1} \circ {T_y}_*\right )([X  \xrightarrow{\op{id}_X}  X]).$$
Here $\bH_{**}(X \xrightarrow{h}  S)$ is the twisted bivariant homology theory with the twisted pushforward $f_{**} :=(-1)^{\op{codim}(f)}f_*$.
\end{cor}
 
\begin{rem}
\begin{enumerate}
\item $\widetilde {\m M {T_y}_*}: K^{\Cal Prop} _{\ell.c.i}(\m V/X \xrightarrow{h}  S)  \to \bH_{**}(X \xrightarrow{h}  S)\otimes \bQ[y]$ shall be called \emph{a motivic ``bivariant" Hirzebruch--Milnor class}, even thought the source group $K^{\Cal Prop} _{\ell.c.i}(\m V/X \xrightarrow{h}  S)$ is not a bivariant theory, but the target group $\bH_{**}(X \xrightarrow{h}  S)\otimes \bQ[y]$ is a bivariant theory.
\item Note that when the base variety $S$ is a point, $\widetilde {\m M {T_y}_*}: K^{\Cal Prop} _{\ell.c.i}(\m V/X \xrightarrow{h}  S)  \to \bH_{**}(X \xrightarrow{h}  S)\otimes \bQ[y]$ is the same as $\m M {T_y}_*: K^{\Cal Prop} _{\ell.c.i}(\m V/X)  \to H_{**}^{BM}(X)\otimes \bQ[y]$.\\
\end{enumerate} 
\end{rem}

\begin{pro}\label{Todd-Milnor} In the case when $y = 0$, the Todd--Milnor class $\m M {T_0}_*: K^{\Cal Prop} _{\Cal \ell.c.i}(\m V/X)  \to H_*^{BM}(X)\otimes \bQ$ vanishes on the subgroup generated by $[V  \xrightarrow{p}  X]$ with $V$ being Du Bois varieties:
$$\m M{T_{0}}_* ([V  \xrightarrow{p}  X]) = 0 \quad \text {if $V$ is a Du Bois variety.}$$
\end{pro}
\begin{proof} For a local complete intersection variety $V$ in a smooth variety $M$, we have that 
\begin{align*}
& \m M {T_0}_* ([V  \xrightarrow{p}  X]) \\
& = p_{**} \m M {T_0}_*([V  \xrightarrow{\op{id}}  V]) \\
& = (-1)^{\op{dim}X} p_*\left ( td(T_V)\cap [V]  - {T_0}_*  ([V  \xrightarrow{\op{id}}  V]) \right ) \\
& = (-1)^{\op{dim}X} p_*\left ( td(T_V)\cap [V]  - {T_0}_*(V) \right ).
\end{align*}
If $V$ is a Du Bois variety, it follows from Remark \ref{DuBois} that ${T_0}_*(V) = td_*(\m O_V)$.
On the other hand we observe that it follows from the properties of the Baum--Fulton--MacPherson's Riemann--Roch $td_*: G_0(X) \to H_*^{BM}(X)\otimes \bQ$ (see \cite [Corollary 18.3.1 (b)]{Fulton-book}, or more generally \cite [PART II, \S 0.2 Summary of results]{Fulton-MacPherson}) that for a local complete intersection variety $V$ in a smooth variety $M$ we have
$$td_*(\m O_V) = td(T_V) \cap [V].$$
Therefore, if $V$ is a local complete intersection variety $V$ in a smooth variety $M$ and also a Du Bois variety, then we have
$$\m M {T_0}_* ([V  \xrightarrow{p}  X]) = 0.$$
\end{proof}

\begin{cor} If the base variety $S$ is a $\bQ$-homology manifold, then the motivic bivariant Todd--Milnor class 
$\widetilde {\m M {T_0}_*}: K^{\Cal Prop} _{\ell.c.i}(\m V/X \xrightarrow{h}  S)  \to \bH_{**}(X \xrightarrow{h}  S)\otimes \bQ$ vanishes on the subgroup generated by $[V  \xrightarrow{p}  X]$ with $V$ being Du Bois varieties.
\end{cor} 

\begin{proof} It follows from that for an element $[V  \xrightarrow{p}  X]$ with $V$ being a Du Bois variety $\widetilde {\m M {T_0}_*}([V  \xrightarrow{p}  X]) \bullet [S] = \m M {T_0}_*([V  \xrightarrow{p}  X]) = 0$  and $\bullet [S] : \bH (X \xrightarrow{h}  S)\otimes \bQ \overset{\cong} \to \bH(X \to pt)\otimes \bQ =H_*^{BM}(X)\otimes \bQ$ is an isomorphism when $S$ is a $\bQ$-homology manifold.
\end{proof}

\begin{rem} 
Let us compare with the results in Theorem \ref{unification}. Neither of the following three diagrams commutes in general:
\begin{itemize} 
\item $\underline {y = -1}:$
$$\xymatrix{
& K^{\Cal Prop} _{\ell.c.i}(\m V/X) \ar [dl]_{\ga_F} \ar [dr]^{{T_{-1}}^{FJ}_*} \\
{F(X) } \ar [rr] _{c_*}& &  H_*^{BM}(X)\otimes \bQ.}
$$
\item $\underline {y = 0}:$
$$\xymatrix{
& K^{\Cal Prop} _{\ell.c.i}(\m V/X) \ar [dl]_{\ga_{G_0}} \ar [dr]^{{T_{0}}^{FJ}_*} \\
{G_0(X) } \ar [rr] _{td_*}& &  H_*^{BM}(X)\otimes \bQ.}
$$
\item $\underline {y = 1}:$
$$\xymatrix{
& K^{\Cal Prop} _{\ell.c.i}(\m V/X)  \ar [dl]_{\ga_{\Omega}} \ar [dr]^{{T_1}^{FJ}_*} \\
{\Omega(X) } \ar [rr] _{L_*}& &  H_*^{BM}(X)\otimes \bQ.}
$$
\end{itemize}
Hence it is natural or reasonable to consider the following commutative diagrams with the corresponding Milnor classes and the corresponding looked-for natural transformations
$$\m M {T_{-1}}_*: K^{\Cal Prop} _{\ell.c.i}(\m V/X)  \to H_{**}^{BM}(X)\otimes \bQ, \quad \m Mc_*:F(X) \to H_{**}^{BM}(X) \otimes \bQ$$
$$\m M {T_0}_*: K^{\Cal Prop} _{\ell.c.i}(\m V/X)  \to H_{**}^{BM}(X)\otimes \bQ, \quad \m Mtd_*:G_0(X) \to H_{**}^{BM}(X) \otimes \bQ,$$
$$\m M {T_1}_*: K^{\Cal Prop} _{\ell.c.i}(\m V/X)  \to H_{**}^{BM}(X)\otimes \bQ, \quad \m ML_*:\Omega(X) \to H_{**}^{BM}(X) \otimes \bQ:$$

\begin{itemize} 
\item $\underline {y = -1}:$
$$\xymatrix{
& K^{\Cal Prop} _{\ell.c.i}(\m V/X)\ar [dl]_{\ga_F} \ar [dr]^{\m M {T_{-1}}_*} \\
{F(X) } \ar [rr] _{\m Mc_*}& &  H_{**}^{BM}(X)\otimes \bQ.}
$$
\item $\underline {y = 0}:$
 $$\xymatrix{
& K^{\Cal Prop} _{\ell.c.i}(\m V/X)\ar [dl]_{\ga_F} \ar [dr]^{\m M {T_{0}}_*} \\
{G_0(X) } \ar [rr] _{\m Mtd_*}& &  H_{**}^{BM}(X)\otimes \bQ.}
$$
\item $\underline {y = 1}:$
 $$\xymatrix{
& K^{\Cal Prop} _{\ell.c.i}(\m V/X)\ar [dl]_{\ga_F} \ar [dr]^{\m M {T_{1}}_*} \\
{\Omega(X) } \ar [rr] _{\m ML_*}& &  H_{**}^{BM}(X)\otimes \bQ.}
$$
\end{itemize}
\end{rem}

\section{Verdier-type Riemann--Roch formulas}
In this section we show \emph{Verdier-type Riemann--Roch formulas}.

First we show \emph {a Verdier-type Riemann--Roch formula for the motivic canonical $c\ell$ class for a smooth morphism}. Here we emphasize that we need a smooth morphism instead of a local complete intersection morphism:

\begin{pro}\label{VRR-HFJ} Let $f:X \to Y$ be a smooth morphism in the category $\m V_S$:
$$
\xymatrix{
X \ar[dr]_ {h}\ar[rr]^ {f} && Y \ar[dl]^{k }\\
& S.}
$$
Then the following diagram commutes:
$$\CD
K^{\Cal Prop}  _{\ell.c.i}(\m V/Y \xrightarrow{k}  S) @> {c\ell_*^{FJ}} >> H_*^{BM}(Y) \otimes R \\
@V f^*VV @VV c\ell(T_f) \cap f^* V\\
K^{\Cal Prop}  _{\ell.c.i}(\m V/X \xrightarrow{h}  S)  @>> {c\ell_*^{FJ}} > H_*^{BM}(X) \otimes R, \endCD
$$
Here $f^*:H_*^{BM}(Y) \to H_*^{BM}(X)$ is the Gysin pullback homomorphism.
\end{pro}

\begin{proof} Let $[W \xrightarrow{p}  Y] \in K^{\Cal Prop}  _{\ell.c.i}(\m V/Y \xrightarrow{k}  S)$ and consider the following diagram whose top square is a fiber square:

\begin{equation}\label{diagram4}
\xymatrix{
W' \ar[rr]^{f'}\ar[d]_{p'} && W \ar[d]^{p}\\
X \ar[dr]_ {h}\ar[rr]^ {f} && Y \ar[dl]^{k }\\
& S.}
\end{equation}
We want to show that 
$$c\ell_*^{FJ}f^*([W \xrightarrow{p}  Y] ) = c\ell(T_f) \cap f^*\left (c\ell_*^{FJ}([W \xrightarrow{p}  Y] ) \right ).$$
\begin{align*}
c\ell_*^{FJ}f^*([W \xrightarrow{p}  Y] ) & = c\ell_*^{FJ}([W' \xrightarrow{p'}  X] )\\
& = p'_*(c\ell (T_{h \circ p'}) \cap [W']) \quad \text {(by definition of $c\ell_*^{FJ}$)}
\end{align*}
$$c\ell(T_f) \cap f^*\left (c\ell_*^{FJ}([W \xrightarrow{p}  Y] ) \right )
= c\ell(T_f) \cap f^* \left (p_*(c\ell (T_{k \circ p}) \cap [W] ) \right ).$$
Since $p:W \to Y$ is proper and $f:X \to Y$ is smooth, hence flat, it follows from \cite[Proposition 1.7]{Fulton-book} that we have the \emph{base change formula}:$f^* p_* = p'_*{f'}^*$. The above equality continues as follows:
\begin{align*}
\hspace{2cm}
& = c\ell(T_f) \cap p'_*{f'}^*(c\ell (T_{k \circ p}) \cap [W] ) \\
& = p'_*\left ({p'}^* c\ell(T_f) \cap {f'}^*(c\ell (T_{k \circ p}) \cap [W] ) \right )\, \,\text {(projection formua)}\\
&= p'_*\left (c\ell({p'}^* T_f) \cap (c\ell ({f'}^*T_{k \circ p}) \cap {f'}^*[W] ) \right )\, \, \text {(by \cite[Theorem 3.2]{Fulton-book})}\\
&= p'_*\left ((c\ell( T_{f'}) \cup c\ell ({f'}^*T_{k \circ p})) \cap [{f'}^{-1}(W)] ) \right )\, \, \text {(by \cite[Lemma1.7.1]{Fulton-book})}\\
&= p'_*\left (c\ell( T_{f'}+{f'}^*T_{k \circ p}) \cap [W'] \right )\\
&= p'_*\left (c\ell( T_{k\circ p \circ f'}) \cap [W'] \right )\quad \text {($T_{k\circ p \circ f'} =T_{f'}+{f'}^*T_{k \circ p} \in K^0(W')$ )}\\
&= p'_*\left (c\ell( T_{h \circ p'}) \cap [W'] \right ).
\end{align*}
Therefore we get that $c\ell_*^{FJ}f^*([W \xrightarrow{p}  Y] ) = c\ell(T_f) \cap f^*\left (c\ell_*^{FJ}([W \xrightarrow{p}  Y] \right ).$
\end{proof}

Secondly we show \emph {a Verdier-type Riemann--Roch formula for the motivic Hirzebruch class for a smooth morphism}:

\begin{pro} Let $f:X \to Y$ be a smooth morphism in the category $\m V_S$ as in Proposition \ref{VRR-HFJ}.
Then the following diagram commutes:
$$\CD
K^{\Cal Prop}  _{\ell.c.i}(\m V/Y \xrightarrow{k}  S) @> {{T_y}_*} >> H_*^{BM}(Y) \otimes \bQ[y]\\
@V f^*VV @VV td_{(y)}(T_f) \cap f^* V\\
K^{\Cal Prop}  _{\ell.c.i}(\m V/X \xrightarrow{h}  S)  @>> {{T_y}_*} > H_*^{BM}(X) \otimes \bQ[y]. \endCD
$$
\end{pro}

\begin{proof} For the above diagram (\ref{diagram4})
we want to show that 
$${T_y}_*f^*([W \xrightarrow{p}  Y] ) = td_{(y)}(T_f) \cap f^*\left ({T_y}_*([W \xrightarrow{p}  Y] ) \right ).
$$
Since it follows from Hironaka's resolution of singularities that any $[W \xrightarrow{p}  Y]$ can be expressed as a linear combination
$$\sum_V a_V [V \xrightarrow{p_V}  Y] $$
where $a_V \in \bZ$, $V$ is a smooth variety , and $p_V :V \to Y$ is proper, if suffices to show that
$${T_y}_*f^*([V \xrightarrow{p_V}  Y]) = td_{(y)}(T_f) \cap f^*\left ({T_y}_*([V \xrightarrow{p_V}  Y]) \right ).$$
Hence, from the beginning we can assume that in the above diagram \ref{diagram4} $W$ is smooth and $p:W \to Y$ is proper, but here note that we DO NOT need the requirement that the composite $k \circ p: W \to S$ is a local complete intersection morphism. Here it should be noted that since $W$ is smooth and $f':W' \to W$ is smooth (because $f'$ is the pullback of the smooth morphism $f:X \to Y$), $W'$ is also smooth, which is crucial below.
\begin{align*}
{T_y}_*f^*([W \xrightarrow{p}  Y] ) & = {T_y}_*([W' \xrightarrow{p'}  X] )\\
& = {T_y}_*(p'_*[W' \xrightarrow{\op{id}_{W'}}  W'] )\\
& = p'_*{T_y}_*([W' \xrightarrow{\op{id}_{W'}}  W'] ) \\
& = p'_*(td_{(y)}(TW') \cap [W']) \quad \text {(since $W'$ is smooth)}.
\end{align*}
On the other hand we have
\begin{align*}
& td_{(y)}(T_f) \cap f^* {T_y}_*([W \xrightarrow{p}  Y] ) \\
& = td_{(y)}(T_f) \cap f^*{T_y}_*(p_*[W \xrightarrow{\op{id}_{W}}  W] )\\
& = td_{(y)}(T_f) \cap f^*p_* ({T_y}_*([W \xrightarrow{\op{id}_{W}}  W] ) ) \\
& = td_{(y)}(T_f) \cap f^*p_* (td_{(y)}(TW) \cap [W])) \quad \text {(since $W$ is smooth)}\\
& = td_{(y)}(T_f) \cap p'_*{f'}^* (td_{(y)}(TW) \cap [W])) \\
& = p'_*\left ({p'}^*td_{(y)}(T_f) \cap {f'}^* (td_{(y)}(TW) \cap [W])\right ) \\
& = p'_*\left (td_{(y)}({p'}^*T_f) \cap \left ({f'}^* td_{(y)}(TW) \cap {f'}^* [W] \right )\right ) \\
\end{align*}
\begin{align*}
& = p'_*\left (\left (td_{(y)}(T_{f'}) \cup td_{(y)}({f'}^*TW) \right) \cap [{f'}^{-1} W])\right ) \\
& = p'_*\left (td_{(y)}(T_{f'} +{f'}^*TW) \cap [W'])\right ) \\
& = p'_*\left (td_{(y)}(TW') \cap [W'])\right ) \quad \text {(since $T_{f'} = TW' - {f'}^*TW$)}.
\end{align*}
Therefore we get that ${T_y}_*f^*([W \xrightarrow{p}  Y] ) = td_{(y)}(T_f) \cap f^*\left ({T_y}_*([W \xrightarrow{p}  Y] ) \right ).$
\end{proof}

\begin{rem} The above proof of course implies that the following Verdier-type Riemann--Roch formula holds for the motivic Hirzebruch class ${T_y}_*: K_0(\m V/X) \to H_*^{BM}(X) \otimes \bQ[y]$: for a smooth morphism $f:X \to Y$ in the category $\m V$ the following diagram commutes:
$$\CD
K_0(\m V/Y)  @> {{T_y}_*} >> H_*^{BM}(Y) \otimes \bQ[y]\\
@V f^*VV @VV td_{(y)}(T_f) \cap f^* V\\
K_0(\m V/X)   @>> {{T_y}_*} > H_*^{BM}(X) \otimes \bQ[y]. \endCD
$$
\end{rem} 

\begin{defn} For a smooth morphism $f:X \to Y$, the twisted Gysin pullback homomophism $f^{**}:H_*^{BM}(Y) \to H_*^{BM}(X)$ is defined by
$$f^{**} = (-)^{\op {dim} (f)} f^* = (-1)^{\op{dim} X - \op {dim} Y} f^*.$$
(In other words, $(-)^{\op {codim} (f)}f^{**} = (-1)^{\op{dim} Y - \op {dim} X}f^{**} = f^*$.)
The contravariant Borel--Moore homology theory with this twisted pullback homomotphism for smoth morphisms is denoted by $H_{**}^{BM}$.
\end{defn}

In \cite[Theorem 2.2]{Yokura-topology} we obtained a Verdier-type Riemann--Roch formula of the Milnor class in a special case.  The following Verdier-type Riemann--Roch formula of the motivc Hirzebruch--Milnor class is a generalization of this result:

\begin{thm} \label{VRR-HM} For a smooth morphism $f: X \to Y$ in the category $\m V_S$ as in Proposition \ref{VRR-HFJ},
 the following diagram commutes:
$$\CD
K^{\Cal Prop}  _{\ell.c.i}(\m V/Y \xrightarrow{k}  S) @> {\m M{T_y}_*} >> H_{**}^{BM}(Y) \otimes \bQ[y] \\
@V f^*VV @VV td_{(y)}(T_f) \cap f^{**} V \\
K^{\Cal Prop}  _{\ell.c.i}(\m V/X \xrightarrow{h}  S)  @>> {\m M{T_y}_*} > H_{**}^{BM}(X) \otimes \bQ[y].\endCD
$$
\end{thm}
\begin{proof} Let $[W \xrightarrow{p}  Y] \in K^{\Cal Prop}  _{\ell.c.i}(\m V/Y \xrightarrow{k}  S)$. Then we have that
\begin{align*}
& \m M{T_y}_*f^*([W \xrightarrow{p}  Y] ) \\
& = \m M{T_y}_*[W'\xrightarrow{p'}  X] ) \\
& = (-1)^{\op {dim} W'} \left ({T_y}_*^{FJ} - {T_y}_* \right )([W'\xrightarrow{p'}  X]) \\
& = (-1)^{\op {dim} W'} \left ({T_y}_*^{FJ} - {T_y}_* \right )(f^*[W \xrightarrow{p}  Y]) \\
& = (-1)^{\op {dim} W'} \left ({T_y}_*^{FJ}f^* - {T_y}_* f^*\right )([W \xrightarrow{p}  Y]) \\
& = (-1)^{\op {dim} W'}  \left (td_{(y)}(T_f) \cap f^*{T_y}_*^{FJ} - td_{(y)}(T_f) \cap f^*{T_y}_* \right )([W \xrightarrow{p}  Y]) \\
& = (-1)^{\op {dim} W'} td_{(y)}(T_f) \cap f^* \left ({T_y}_*^{FJ} - {T_y}_* \right )([W \xrightarrow{p}  Y]) \\
\end{align*}
\begin{align*}
& = (-1)^{\op {dim} W'} (-)^{\op {codim} (f)} td_{(y)}(T_f) \cap f^{**} \left ({T_y}_*^{FJ} - {T_y}_* \right )([W \xrightarrow{p}  Y]) \\
& = (-1)^{\op {dim} W' + \op{dim} Y - \op {dim} X} td_{(y)}(T_f) \cap f^{**} \left ({T_y}_*^{FJ} - {T_y}_* \right )([W \xrightarrow{p}  Y]) \\
& = (-1)^{\op {dim} W} td_{(y)}(T_f) \cap f^{**} \left ({T_y}_*^{FJ} - {T_y}_* \right )([W \xrightarrow{p}  Y]) \\
& = td_{(y)}(T_f) \cap f^{**} \left ((-1)^{\op {dim} W} \left ({T_y}_*^{FJ} - {T_y}_* \right )([W \xrightarrow{p}  Y]) \right ) \\
& = td_{(y)}(T_f) \cap f^{**} \left (\m M{T_y}_*([W \xrightarrow{p}  Y] ) \right ).
\end{align*}
\end{proof}

Finally we give a ``bivariant version" of Theorem \ref{VRR-HM}:
\begin{cor} For a smooth morphism $f: X \to Y$ in the category $\m V_S$ as in Proposition \ref{VRR-HFJ}, 
 the following diagram commutes:
$$\CD
K^{\Cal Prop}  _{\ell.c.i}(\m V/Y \xrightarrow{k}  S) @> {\widetilde {\m M{T_y}_*}} >> \bH(Y \xrightarrow{h}  S) \otimes \bQ[y] \\
@V f^*VV @VV (-1)^{\op{dim}(f)} td_{(y)}(T_f) \bullet U_f \bullet V \\
K^{\Cal Prop}  _{\ell.c.i}(\m V/X \xrightarrow{h}  S)  @>> {\widetilde {\m M{T_y}_*}} > \bH(X \xrightarrow{h}  S)\otimes \bQ[y] , \endCD
$$
\end{cor}
\begin{proof} The commutativity of the above diagram follows from Theorem \ref{VRR-HM}, the following commutative diagram
$$\CD
\bH(Y \xrightarrow{k}  S)\otimes \bQ[y]  @> {\bullet [S] } >> \bH_*^{BM}(Y) \otimes \bQ[y] \\
@V {(-1)^{\op{dim}(f)} td_{(y)}(T_f) \bullet U_f \bullet} VV @VV td_{(y)}(T_f) \cap f^{**} V \\
\bH(X \xrightarrow{h}  S)\otimes \bQ[y] @>> {\bullet [S]} > \bH_*^{BM}(X) \otimes \bQ[y] , \endCD
$$
and the fact (see \cite{Fulton-MacPherson}) that for any $\beta \in \bH(Y \to pt) = H_*^{BM}(Y)$
$$U_f \bullet \beta = f^*\beta$$
and also using the fact that $\bullet [S]: \bH(X \xrightarrow{h}  S) \overset {\cong}\longrightarrow  H_*^{BM}(X)$ is an isomorphism.
\end{proof} 

{\bf Acknowledgements:}  The author thanks Laurentiu Maxim  and J\"org Sch\"urmann for sending him their preprint \cite{CMSS} and some useful comments, and also Paolo Aluffi for some comments.

\end{document}